\documentclass{amsart}
\usepackage{amssymb}
\usepackage{amsfonts}
\usepackage{geometry}

\newtheorem{theorem}{Theorem}
\theoremstyle{plain}
\newtheorem{acknowledgement}{Acknowledgement}

\newtheorem{corollary}{Corollary}

\newtheorem{example}{Example}

\newtheorem{lemma}{Lemma}

\newtheorem{proposition}{Proposition}
\newtheorem{remark}{Remark}

\numberwithin{equation}{section}
\input{tcilatex}
\begin{document}
\title{The $\theta $-bump theorem for product fractional integrals}
\author{Eric Sawyer}
\address{McMaster University}
\author{Zipeng Wang}
\address{McMaster University}

\begin{abstract}
We extend the one parameter $\theta $-bump theorem for fractional integrals
of Sawyer and Wheeden to the setting of two parameters, as well as improving
the multiparameter result of Tanaka and Yabuta for doubling weights to
classical reverse doubling weights.
\end{abstract}

\maketitle
\tableofcontents

\section{Introduction}

In \cite[Theorem 1(A)]{SaWh}, Sawyer and Wheeden proved that the fractional
integral $I_{\alpha }f\left( x\right) =\int \left\vert x-u\right\vert
^{\alpha -n}f\left( u\right) du$, $x\in \mathbb{R}^{N}$, is bounded from one
weighted space $L^{p}\left( v^{p}\right) $ to another $L^{q}\left(
w^{q}\right) $ provided there is $\theta >1$ such that%
\begin{equation*}
A_{p,q;\theta }^{\alpha ,m}\left( v,w\right) \equiv \sup_{I\in \mathcal{D}%
^{N}}\left\vert I\right\vert ^{\frac{\alpha }{m}-\frac{1}{p}+\frac{1}{q}%
}\left( \frac{1}{\left\vert I\right\vert }\int_{I}v^{-p^{\prime }\theta
}\right) ^{\frac{1}{p^{\prime }\theta }}\ \left( \frac{1}{\left\vert
I\right\vert }\int_{I}w^{q\theta }\right) ^{\frac{1}{q\theta }}<\infty .
\end{equation*}%
Here $1<p\leq q<\infty $, $0<\alpha <N$ and $v,w$ are nonnegative measurable
functions on $\mathbb{R}^{N}$, $N\geq 1$. The finiteness of $%
A_{p,q;1}^{\alpha ,m}\left( v,w\right) $ when $\theta =1$ is a well-known
necessary condition for the boundedness of $I_{\alpha }$, and the above
strengthening of that condition is usually referred to as a $\theta $-bump
condition. In the same paper \cite[the second assertion of Theorem 1(B)]%
{SaWh}, it was shown that in the case $p<q$, if $v^{-p^{\prime }}$ and $%
w^{q} $ are both reverse doubling weights, then the necessary condition $%
A_{p,q;1}^{\alpha ,m}\left( v,w\right) <\infty $ is also sufficient for the
boundedness of $I_{\alpha }$ from $L^{p}\left( v^{p}\right) $ to $%
L^{q}\left( w^{q}\right) $. Here a measure $\mu $ is \emph{reverse doubling}
in $\mathbb{R}^{N}$ if there are $C,\varepsilon >0$ such that%
\begin{equation*}
\left\vert 2^{-s}I\right\vert _{\mu }\leq C2^{-\varepsilon s}\left\vert
I\right\vert _{\mu }\ ,\ \ \ \ \ \text{for all }s>0\text{ and cubes }%
I\subset \mathbb{R}^{N},
\end{equation*}%
where $2^{-s}I$ denotes the cube \emph{concentric} with $I$ and having side
length $\ell \left( 2^{-s}I\right) $ equal to $2^{-s}\ell \left( I\right) $.

Recently, H. Tanaka and K. Yabuta \cite{TaYa} used a clever iteration%
\footnote{%
In a nutshell, they use $p<r<q$ and H\"{o}lder's inequality with $r$ and $%
r^{\prime }$ to separate the measures $\sigma $ and $\omega $ early on, and
then use iteration on the resulting `one weight' Carleson embeddings, the
point being that iteration works better with one weight than with two
weights.} to obtain an $n$-linear embedding theorem for rectangles that has
as a corollary the following result for certain \emph{product} fractional
integrals $\widetilde{I}_{\alpha }^{N}$ on $\mathbb{R}^{N}$ given by 
\begin{equation*}
\widetilde{I}_{\alpha }^{N}f\left( x\right) \equiv \int_{\mathbb{R}%
^{N}}\prod_{j=1}^{N}\left\vert x_{j}-u_{j}\right\vert ^{\alpha -1}f\left(
u\right) du,\ \ \ \ \ x\in \mathbb{R}^{N},\ \ \ \ \ 0<\alpha <1.
\end{equation*}%
Let $\mathcal{R}^{N}$ denote the \emph{partial grid} of all rectangles in $%
\mathbb{R}^{N}$ with sides parallel to the coordinate axes (which is not a
grid). A weight $\mu $ is a rectangle doubling weight on $\mathbb{R}^{N}$ if
there is $C>0$ such that%
\begin{equation*}
\left\vert 2R\right\vert _{\mu }\leq C\left\vert R\right\vert _{\mu },\ \ \
\ \text{ for all rectangles }R\in \mathcal{R}^{N}.
\end{equation*}

\begin{theorem}[{H. Tanaka and K. Yabuta \protect\cite[Proposition 5.1]{TaYa}%
}]
Suppose $1<p<q<\infty $ and that both $v^{-p^{\prime }}$ and $w^{q}$ are
rectangle doubling weights\footnote{%
In \cite{TaYa} the authors use a strong form of reverse doubling on
rectangles, which is equivalent to rectangle doubling. See the appendix
below.} on $\mathbb{R}^{N}$. Then $\widetilde{I}_{\alpha }^{N}$ is bounded
from $L^{p}\left( v^{p}\right) $ to $L^{q}\left( w^{q}\right) $ if and only
if 
\begin{equation*}
\sup_{R\in \mathcal{R}^{N}}\left\vert R\right\vert ^{\frac{\alpha }{N}-\frac{%
1}{p}+\frac{1}{q}}\left( \frac{1}{\left\vert R\right\vert }%
\int_{R}v^{-p^{\prime }}\right) ^{\frac{1}{p^{\prime }}}\ \left( \frac{1}{%
\left\vert R\right\vert }\int_{R}w^{q}\right) ^{\frac{1}{q}}<\infty .
\end{equation*}
\end{theorem}

Thus the theorem of Tanaka and Yabuta extends the second assertion in
Theorem 1(B) of \cite{SaWh} to product fractional integrals with rectangle
doubling weights. The purpose of this paper is to extend both Theorem 1(A)
and the second assertion in Theorem 1(B) of \cite{SaWh} to product
fractional integrals of the form (more than two factors in the kernel are
handled similarly)%
\begin{equation*}
I_{\alpha ,\beta }^{m,n}f\left( x,y\right) \equiv \int_{\mathbb{R}^{m}}\int_{%
\mathbb{R}^{n}}\left\vert x-u\right\vert ^{\frac{\alpha }{m}-1}\left\vert
y-t\right\vert ^{\frac{\beta }{n}-1}f\left( u,t\right) dtdu,\ \ \ \ \ \left(
x,y\right) \in \mathbb{R}^{m}\times \mathbb{R}^{n},
\end{equation*}%
more precisely, to show that a product $\theta $-bump condition is always
sufficient for the norm inequality, and that the same condition without a
bump is sufficient provided the weights $v^{-p^{\prime }}$ and $w^{q}$ are
product reverse doubling on $\mathbb{R}^{m}\times \mathbb{R}^{n}$ in this
sense: a weight $\mu $ is \emph{product reverse doubling} on $\mathbb{R}%
^{m}\times \mathbb{R}^{n}$ if there are $C,\varepsilon _{1},\varepsilon
_{2}>0$ such that%
\begin{equation}
\left\vert \left( 2^{-s}I\right) \times \left( 2^{-t}J\right) \right\vert
_{\mu }\leq C2^{-\varepsilon _{1}s-\varepsilon _{2}t}\left\vert I\times
J\right\vert _{\mu }\ ,\ \ \ \ \ \text{for all }s,t>0\text{ and cubes }%
I\subset \mathbb{R}^{m}\text{ and }J\subset \mathbb{R}^{n}.
\label{product rev doub}
\end{equation}

Our proof of the first result adapts the Tanaka-Yabuta argument to the $%
\theta $-bump functional used in \cite{SaWh}, while the second result
regarding reverse doubling weights adapts the Tanaka-Yabuta argument to the
use of NTV good/bad grids in place of the Str\"{o}mberg $\frac{1}{3}$-trick
that was used in \cite{TaYa}. Additional results for the product situation
can be found in our paper \cite{SaWa}. See the appendix below for a
discussion of the doubling and various reverse doubling conditions.

\begin{acknowledgement}
We are grateful to Hitoshi Tanaka for bringing our attention to his
beautiful paper \cite{TaYa} with K. Yabuta.
\end{acknowledgement}

\subsection{Preliminaries}

Let $\mathcal{D}^{m}$ denote the grid of dyadic cubes in $\mathbb{R}^{m}$,
and let $\mathcal{R}^{m,n}\equiv \mathcal{D}^{m}\times \mathcal{D}^{n}$
denote the partial grid of dyadic rectangles in $\mathbb{R}^{m}\times 
\mathbb{R}^{n}$ (which is not actually a grid since it fails the nested
property). For $d\mu \left( x\right) =u\left( x\right) dx$ absolutely
continuous with respect to Lebesgue measure on $\mathbb{R}^{N}$, we will use
the following $\theta $-bump functional for a cube $Q$ and $\theta >1$ (\cite%
[see page 830]{SaWh}):%
\begin{equation*}
\left\vert Q\right\vert _{\mu ,\theta }\equiv \left\vert Q\right\vert ^{%
\frac{1}{\theta ^{\prime }}}\left( \int_{Q}u^{\theta }\right) ^{\frac{1}{%
\theta }}.
\end{equation*}%
We have $\left\vert Q\right\vert _{\mu }\leq \left\vert Q\right\vert _{\mu
,\theta }$, and if $P=\overset{\cdot }{\bigcup }_{i=1}^{\infty }Q_{i}$ is a
pairwise disjoint union of the cubes $Q_{i}$, then we have%
\begin{equation*}
\sum_{i=1}^{\infty }\left\vert Q_{i}\right\vert _{\mu ,\theta
}=\sum_{i=1}^{\infty }\left\vert Q\right\vert ^{\frac{1}{\theta ^{\prime }}%
}\left( \int_{Q}u^{\theta }\right) ^{\frac{1}{\theta }}\leq \left(
\sum_{i=1}^{\infty }\left\vert Q_{i}\right\vert \right) ^{\frac{1}{\theta
^{\prime }}}\left( \sum_{i=1}^{\infty }\int_{Q_{i}}u^{\theta }\right) ^{%
\frac{1}{\theta }}=\left\vert P\right\vert ^{\frac{1}{\theta ^{\prime }}%
}\left( \int_{P}u^{\theta }\right) ^{\frac{1}{\theta }}=\left\vert
P\right\vert _{\mu ,\theta }\ .
\end{equation*}

The important property of the $\theta $-bump functional on cubes for us is
that, when taken to a power larger than $1$, it automatically satisfies a
Carleson condition taken over all dyadic subcubes. More precisely, if $\rho
>1$, then%
\begin{eqnarray}
\sum_{Q\in \mathcal{D}^{N}:\ Q\subset P}\left\vert Q\right\vert _{\mu
,\theta }^{\rho } &=&\sum_{k=-\infty }^{\infty }\sum_{Q\in \mathcal{D}^{N}:\
\ell \left( Q\right) =2^{-k}\ell \left( P\right) }\left\vert Q\right\vert ^{%
\frac{\rho -1}{\theta ^{\prime }}}\left( \int_{Q}u^{\theta }\right) ^{\frac{%
\rho -1}{\theta }}\left\vert Q\right\vert _{\mu ,\theta }  \label{automatic}
\\
&\leq &\sum_{k=-\infty }^{\infty }\sum_{Q\in \mathcal{D}^{N}:\ \ell \left(
Q\right) =2^{-k}\ell \left( P\right) }\left( C2^{-kN\varepsilon }\left\vert
P\right\vert \right) ^{\frac{\rho -1}{\theta ^{\prime }}}\left(
\int_{P}u^{\theta }\right) ^{\frac{\rho -1}{\theta }}\left\vert Q\right\vert
_{\mu ,\theta }  \notag \\
&\leq &\sum_{k=-\infty }^{\infty }\left( C2^{-kN\varepsilon }\left\vert
P\right\vert \right) ^{\frac{\rho -1}{\theta ^{\prime }}}\left(
\int_{P}u^{\theta }\right) ^{\frac{\rho -1}{\theta }}\left\vert P\right\vert
_{\mu ,\theta }=C_{N\varepsilon \frac{\rho -1}{\theta ^{\prime }}}\left\vert
P\right\vert _{\mu ,\theta }^{\rho }\ .  \notag
\end{eqnarray}%
This automatic Carleson condition leads to a corresponding automatic
Carleson embedding lemma.

\begin{lemma}
\label{theta bump lemma}Suppose that $1<s<r<\infty $, $\theta >1$, and that $%
d\mu \left( x\right) =u\left( x\right) dx$ is a locally $L^{\theta }$
absolutely continuous measure on $\mathbb{R}^{N}$. Then we have%
\begin{equation*}
\left\{ \sum_{Q\in \mathcal{D}^{N}}\left\vert Q\right\vert _{\mu ,\theta }^{%
\frac{r}{s}}\left( \frac{1}{\left\vert Q\right\vert _{\mu ,\theta }}%
\int_{Q}fd\mu \right) ^{r}\right\} ^{\frac{1}{r}}\leq C_{r,s,\theta
}\left\Vert f\right\Vert _{L^{s}\left( \mu \right) }\ ,\ \ \ \ \ f\geq 0.
\end{equation*}
\end{lemma}

\begin{proof}
The cubes in $\mathcal{D}^{N}$ form a grid, and so for each integer $k\in 
\mathbb{Z}$, we can consider the maximal dyadic cubes $\left\{
M_{i}^{k}\right\} _{i=1}^{\infty }$ from $\mathcal{D}^{N}$ such that%
\begin{equation*}
\frac{1}{\left\vert M_{i}^{k}\right\vert _{\mu ,\theta }}\int_{M_{i}^{k}}fd%
\mu >2^{k}.
\end{equation*}%
Then we can estimate using (\ref{automatic}) that%
\begin{eqnarray*}
&&\sum_{Q\in \mathcal{D}^{N}}\left\vert Q\right\vert _{\mu ,\theta }^{\frac{r%
}{s}}\left( \frac{1}{\left\vert Q\right\vert _{\mu ,\theta }}\int_{Q}fd\mu
\right) ^{r}\leq \sum_{k=-\infty }^{\infty }\sum_{\substack{ Q\in \mathcal{D}%
^{N}  \\ 2^{k}<\frac{1}{\left\vert Q\right\vert _{\mu ,\theta }}%
\int_{Q}fd\mu \leq 2^{k+1}}}\left\vert Q\right\vert _{\mu ,\theta }^{\frac{r%
}{s}}\left( 2^{k+1}\right) ^{r} \\
&\leq &\sum_{k=-\infty }^{\infty }\ \sum_{i=1}^{\infty }\sum_{Q\in \mathcal{D%
}^{N}:\ Q\subset M_{i}^{k}}\left\vert Q\right\vert _{\mu ,\theta }^{\frac{r}{%
s}}\left( 2^{k+1}\right) ^{r} \\
&=&2^{r}\sum_{k=-\infty }^{\infty }\ \sum_{i=1}^{\infty }\left\{ \sum_{Q\in 
\mathcal{D}^{N}:\ Q\subset M_{i}^{k}}\left\vert Q\right\vert _{\mu ,\theta
}^{\frac{r}{s}}\right\} 2^{kr}\leq 2^{r}C_{N\varepsilon \frac{\frac{r}{s}-1}{%
\theta ^{\prime }}}\sum_{k=-\infty }^{\infty }\ \sum_{i=1}^{\infty
}\left\vert M_{i}^{k}\right\vert _{\mu ,\theta }^{\frac{r}{s}}2^{kr}\ .
\end{eqnarray*}%
Now we use the fact that 
\begin{eqnarray*}
\frac{1}{\left\vert M_{i}^{k}\right\vert _{\mu ,\theta }}\int_{M_{i}^{k}\cap
\left\{ f>2^{k-1}\right\} }fd\mu &=&\frac{1}{\left\vert M_{i}^{k}\right\vert
_{\mu ,\theta }}\int_{M_{i}^{k}}fd\mu -\frac{1}{\left\vert
M_{i}^{k}\right\vert _{\mu ,\theta }}\int_{M_{i}^{k}\cap \left\{ f\leq
2^{k-1}\right\} }fd\mu \\
&\geq &\frac{1}{\left\vert M_{i}^{k}\right\vert _{\mu ,\theta }}%
\int_{M_{i}^{k}}fd\mu -\frac{1}{\left\vert M_{i}^{k}\right\vert _{\mu
,\theta }}\int_{M_{i}^{k}}2^{k-1}d\mu \\
&>&2^{k}-2^{k-1}\frac{\left\vert M_{i}^{k}\right\vert _{\mu }}{\left\vert
M_{i}^{k}\right\vert _{\mu ,\theta }}\geq 2^{k-1},
\end{eqnarray*}%
to obtain%
\begin{eqnarray*}
\sum_{Q\in \mathcal{D}^{N}}\left\vert Q\right\vert _{\mu ,\theta }^{\frac{r}{%
s}}\left( \frac{1}{\left\vert Q\right\vert _{\mu ,\theta }}\int_{Q}fd\mu
\right) ^{r} &\leq &2^{r}C_{N\varepsilon \frac{\frac{r}{s}-1}{\theta
^{\prime }}}\sum_{k=-\infty }^{\infty }\sum_{i=1}^{\infty }\left\vert
M_{i}^{k}\right\vert _{\mu ,\theta }^{\frac{r}{s}}2^{kr} \\
&\leq &C_{r,s,\theta }^{r}\sum_{k=-\infty }^{\infty }\sum_{i=1}^{\infty
}\left( 2^{-k}\int_{M_{i}^{k}\cap \left\{ f>2^{k-1}\right\} }fd\mu \right) ^{%
\frac{r}{s}}2^{kr} \\
&\leq &C_{r,s,\theta }^{r}\left( \sum_{k=-\infty }^{\infty
}\sum_{i=1}^{\infty }2^{k\left( s-1\right) }\int_{M_{i}^{k}\cap \left\{
f>2^{k-1}\right\} }fd\mu \right) ^{\frac{r}{s}}.
\end{eqnarray*}%
We now use that the cubes $\left\{ M_{i}^{k}\right\} _{i=1}^{\infty }$ are
pairwise disjoint in $i$ to continue with the estimate%
\begin{eqnarray*}
\left( \sum_{k=-\infty }^{\infty }\sum_{i=1}^{\infty }2^{k\left( s-1\right)
}\int_{M_{i}^{k}\cap \left\{ f>2^{k-1}\right\} }fd\mu \right) ^{\frac{r}{s}}
&\leq &\left( \sum_{k=-\infty }^{\infty }2^{k\left( s-1\right)
}\int_{\left\{ f>2^{k-1}\right\} }fd\mu \right) ^{\frac{r}{s}} \\
&=&\left( \int \left\{ \sum_{k\in \mathbb{Z}:\ 2^{k}<2f\left( x\right)
}2^{k\left( s-1\right) }\right\} f\left( x\right) d\mu \left( x\right)
\right) ^{\frac{r}{s}} \\
&\leq &C_{s}\left( \int f\left( x\right) ^{\left( s-1\right) }f\left(
x\right) d\mu \left( x\right) \right) ^{\frac{r}{s}} \\
&=&C_{s}\left( \int f\left( x\right) ^{s}d\mu \left( x\right) \right) ^{%
\frac{r}{s}}=C_{s}\left\Vert f\right\Vert _{L^{s}\left( \mu \right) }^{r}\ .
\end{eqnarray*}
\end{proof}

\section{The $2$-parameter theory}

Here we state and prove our extensions of Theorem 1(A) and the second
assertion of Theorem 1(B) in \cite{SaWh}. We begin with the $\theta $-bump
condition.

\subsection{The $\protect\theta $-bump condition for bilinear embeddings}

Here is a variation on the Tanaka-Yabuta theorem \cite[Theorem 1.1]{TaYa}
involving general weights that satisfy a $\theta $-bump analogue of the
`rectangle testing' condition in \cite{TaYa}. We extend the definition of
the $\theta $-bump functional to rectangles in the obvious way,%
\begin{equation*}
\left\vert R\right\vert _{\mu ,\theta }\equiv \left\vert R\right\vert ^{%
\frac{1}{\theta ^{\prime }}}\left( \int_{R}u^{\theta }\right) ^{\frac{1}{%
\theta }},
\end{equation*}%
for $d\mu \left( x,y\right) =u\left( x,y\right) dxdy$ absolutely continuous
and $R$ a rectangle in $\mathbb{R}^{m}\times \mathbb{R}^{n}$.

\begin{theorem}
\label{variation}Suppose $1<p<q<\infty $. Let $d\sigma =v^{-p^{\prime }}dx$
and $d\omega =w^{q}dx$ be locally finite absolutely continuous weights on $%
\mathbb{R}^{m}\times \mathbb{R}^{n}$, let $\theta >1$, and let $K:\mathcal{R}%
^{m,n}\rightarrow \left[ 0,\infty \right) $. Then the norm $\mathbb{N}%
_{K}\left( \sigma ,\omega \right) $ of the positive bilinear inequality,%
\begin{equation*}
\sum_{R\in \mathcal{R}^{m,n}}K\left( R\right) \left( \int_{R}fd\sigma
\right) \left( \int_{R}gd\omega \right) \leq \mathbb{N}_{K}\left( \sigma
,\omega \right) \ \left\Vert f\right\Vert _{L^{p}\left( \sigma \right)
}\left\Vert g\right\Vert _{L^{q^{\prime }}\left( \omega \right) }\ ,\ \ \ \
\ f,g\geq 0,
\end{equation*}%
is finite independent of all partial grids $\mathcal{R}^{m,n}=\mathcal{D}%
^{m}\times \mathcal{D}^{n}$ \emph{if} the $\theta $-bump product
characteristic $\mathbb{A}_{K,\theta }\left( \sigma ,\omega \right) $ is
finite, where%
\begin{eqnarray*}
\mathbb{A}_{K,\theta }\left( \sigma ,\omega \right) &\equiv &\sup_{R\in 
\mathcal{R}^{m,n}}K\left( R\right) \ \left[ \left\vert R\right\vert ^{\frac{1%
}{p^{\prime }\theta ^{\prime }}}\left( \int_{R}v^{-p^{\prime }\theta
}d\sigma \right) ^{\frac{1}{p^{\prime }\theta }}\right] \ \left[ \left\vert
R\right\vert ^{\frac{1}{q\theta ^{\prime }}}\left( \int_{R}w^{q\theta
}d\omega \right) ^{\frac{1}{q\theta }}\right] \\
&=&\sup_{R\in \mathcal{R}^{m,n}}K\left( R\right) \ \left\vert R\right\vert
_{\omega ,\theta }^{\frac{1}{q}}\ \left\vert R\right\vert _{\sigma ,\theta
}^{\frac{1}{p^{\prime }}}\ .
\end{eqnarray*}
\end{theorem}

\begin{proof}
As in \cite{TaYa}, we choose $p<r<q$. Then the definition of the $\theta $%
-bump characteristic, followed by H\"{o}lder's inequality with exponents $r$
and $r^{\prime }$, gives%
\begin{eqnarray*}
&&\sum_{R\in \mathcal{R}^{m,n}}K\left( R\right) \left( \int_{R}fd\sigma
\right) \left( \int_{R}gd\omega \right) \\
&=&\sum_{R\in \mathcal{R}^{m,n}}\left\{ K\left( R\right) \ \left\vert
R\right\vert _{\sigma ,\theta }^{\frac{1}{p^{\prime }}}\left\vert
R\right\vert _{\omega ,\theta }^{\frac{1}{q}}\right\} \left\vert
R\right\vert _{\sigma ,\theta }^{\frac{1}{p}}\left\vert R\right\vert
_{\omega ,\theta }^{\frac{1}{q^{\prime }}}\left( \frac{1}{\left\vert
R\right\vert _{\sigma ,\theta }}\int_{R}fd\sigma \right) \left( \frac{1}{%
\left\vert R\right\vert _{\omega ,\theta }}\int_{R}gd\omega \right) \\
&\leq &\mathbb{A}_{K,\theta }\left( \sigma ,\omega \right) \left\{
\sum_{R\in \mathcal{R}^{m,n}}\left\vert R\right\vert _{\sigma ,\theta }^{%
\frac{r}{p}}\left( \frac{1}{\left\vert R\right\vert _{\sigma ,\theta }}%
\int_{R}fd\sigma \right) ^{r}\right\} ^{\frac{1}{r}}\left\{ \sum_{R\in 
\mathcal{R}^{m,n}}\left\vert R\right\vert _{\omega ,\theta }^{\frac{%
r^{\prime }}{q^{\prime }}}\left( \frac{1}{\left\vert R\right\vert _{\omega
,\theta }}\int_{R}gd\omega \right) ^{r^{\prime }}\right\} ^{\frac{1}{%
r^{\prime }}},
\end{eqnarray*}%
and the theorem now follows from the following proposition.
\end{proof}

\begin{proposition}
Suppose that $1<s<r<\infty $, $\theta >1$, and that $\mu $ is a locally
finite absolutely continuous measure on $\mathbb{R}^{m}\times \mathbb{R}^{n}$%
. Then we have%
\begin{equation*}
\left\{ \sum_{R\in \mathcal{R}^{m,n}}\left\vert R\right\vert _{\mu ,\theta
}^{\frac{r}{s}}\left( \frac{1}{\left\vert R\right\vert _{\mu ,\theta }}%
\int_{R}fd\mu \right) ^{r}\right\} ^{\frac{1}{r}}\leq C_{s,r,\theta
}\left\Vert f\right\Vert _{L^{s}\left( \mu \right) }\ ,\ \ \ \ \ f\geq 0.
\end{equation*}
\end{proposition}

\begin{proof}
We follow the outline of the iteration argument in H. Tanaka and K. Yabuta 
\cite{TaYa}, but adapted to $\theta $-bump functionals. Let $d\mu \left(
x,y\right) =u\left( x,y\right) dxdy$ and define%
\begin{eqnarray*}
&&u^{y}\left( x\right) \equiv u\left( x,y\right) \text{ and }u_{x}\left(
y\right) \equiv u\left( x,y\right) , \\
&&d\mu ^{y}\left( x\right) =u^{y}\left( x\right) dx\text{ and }d\mu
_{x}\left( y\right) =u_{x}\left( y\right) dy \\
&&\ \ \ \ \ \text{for a.e. }x\in \mathbb{R}^{m},\text{ a.e. }y\in \mathbb{R}%
^{n},
\end{eqnarray*}%
and note that%
\begin{eqnarray*}
&&\left\vert J\right\vert _{\mu _{x},\theta }\equiv \left\vert J\right\vert
^{\frac{1}{\theta ^{\prime }}}\left( \int_{J}u_{x}\left( y\right) ^{\theta
}dy\right) ^{\frac{1}{\theta }}\text{ and }\left\vert I\right\vert _{\mu
^{y},\theta }\equiv \left\vert I\right\vert ^{\frac{1}{\theta ^{\prime }}%
}\left( \int_{I}u^{y}\left( y\right) ^{\theta }dx\right) ^{\frac{1}{\theta }}
\\
&&\ \ \ \ \ \text{for a.e. }x\in \mathbb{R}^{m},\text{ a.e. }y\in \mathbb{R}%
^{n}.
\end{eqnarray*}

Now take $f\in L^{p}\left( \mu \right) $ and let%
\begin{equation*}
F^{J}\left( x\right) \equiv \frac{1}{\left\vert J\right\vert _{\mu
_{x},\theta }}\int_{J}f\left( x,y\right) u\left( x,y\right) dy\ \ \ \ \ 
\text{for a.e. }x\in \mathbb{R}^{m}.
\end{equation*}%
Note that%
\begin{eqnarray*}
\left\vert I\times J\right\vert _{\mu ,\theta } &=&\left\vert I\times
J\right\vert ^{\frac{1}{\theta ^{\prime }}}\left( \int_{I}\left\{
\int_{J}u\left( x,y\right) ^{\theta }dy\right\} dx\right) ^{\frac{1}{\theta }%
} \\
&=&\left\vert I\right\vert ^{\frac{1}{\theta ^{\prime }}}\left(
\int_{I}\left\{ \left\vert J\right\vert ^{\frac{1}{\theta ^{\prime }}}\left(
\int_{J}u\left( x,y\right) ^{\theta }dy\right) ^{\frac{1}{\theta }}\right\}
^{\theta }dx\right) ^{\frac{1}{\theta }}
\end{eqnarray*}%
where we can interpret the term in braces as%
\begin{equation*}
\left\vert J\right\vert ^{\frac{1}{\theta ^{\prime }}}\left(
\int_{J}u_{x}\left( y\right) ^{\theta }dy\right) ^{\frac{1}{\theta }%
}=\left\vert J\right\vert _{\mu _{x},\theta }
\end{equation*}%
so that we have%
\begin{equation*}
\left\vert I\times J\right\vert _{\mu ,\theta }=\left\vert I\right\vert ^{%
\frac{1}{\theta ^{\prime }}}\left( \int_{I}\left\vert J\right\vert _{\mu
_{x},\theta }^{\theta }dx\right) ^{\frac{1}{\theta }}\equiv \left\vert
I\right\vert ^{\frac{1}{\theta ^{\prime }}}\left( \int_{I}\left( J_{\mu
,\theta }\left( x\right) \right) ^{\theta }dx\right) ^{\frac{1}{\theta }%
}=\left\vert I\right\vert _{J_{\mu ,\theta },\theta }
\end{equation*}%
where we have defined the absolutely continuous measure $J_{\mu ,\theta }$
by $dJ_{\mu ,\theta }\left( x\right) =J_{\mu ,\theta }\left( x\right) dx$
and where its density function, which with a small abuse of notation we also
denote by $J_{\mu ,\theta }$, is given by 
\begin{equation*}
J_{\mu ,\theta }\left( x\right) \equiv \left\vert J\right\vert _{\mu
_{x},\theta },\ \ \ \ \ x\in \mathbb{R}^{m}.
\end{equation*}%
We then estimate%
\begin{eqnarray*}
&&\sum_{R\in \mathcal{R}^{m,n}}\left\vert R\right\vert _{\mu ,\theta }^{%
\frac{r}{s}}\left( \frac{1}{\left\vert R\right\vert _{\mu ,\theta }}%
\int_{R}f\left( x,y\right) u\left( x,y\right) dxdy\right) ^{r} \\
&=&\sum_{I\times J\in \mathcal{R}^{m,n}}\left\vert I\times J\right\vert
_{\mu ,\theta }^{\frac{r}{s}}\left( \frac{1}{\left\vert I\times J\right\vert
_{\mu ,\theta }}\int_{I\times J}f\left( x,y\right) u\left( x,y\right)
dxdy\right) ^{r} \\
&=&\sum_{J\in \mathcal{D}^{n}}\sum_{I\in \mathcal{D}^{m}}\left\vert
I\right\vert _{J_{\mu ,\theta },\theta }^{\frac{r}{s}}\left( \frac{1}{%
\left\vert I\right\vert _{J_{\mu ,\theta },\theta }}\int_{I}\left(
\int_{J}f\left( x,y\right) u\left( x,y\right) dy\ \frac{1}{J_{\mu ,\theta
}\left( x\right) }\right) J_{\mu ,\theta }\left( x\right) dx\right) ^{r} \\
&=&\sum_{J\in \mathcal{D}^{n}}\left\{ \sum_{I\in \mathcal{D}^{m}}\left\vert
I\right\vert _{J_{\mu ,\theta },\theta }^{\frac{r}{s}}\left( \frac{1}{%
\left\vert I\right\vert _{J_{\mu ,\theta },\theta }}\int_{I}F^{J}\left(
x\right) J_{\mu ,\theta }\left( x\right) dx\right) ^{r}\right\} \lesssim
\sum_{J\in \mathcal{D}^{n}}\left( \int_{\mathbb{R}^{m}}F^{J}\left( x\right)
^{s}J_{\mu ,\theta }\left( x\right) dx\right) ^{\frac{r}{s}},
\end{eqnarray*}%
by Lemma \ref{theta bump lemma}\ above applied with the locally finite
absolutely continuous measures $J_{\mu ,\theta }$ on $\mathbb{R}^{m}$, $J\in 
\mathcal{D}^{n}$. Now we continue to estimate the latter sum raised to the
power $\frac{s}{r}$ by Minkowski's inequality,%
\begin{equation*}
\left\{ \sum_{J\in \mathcal{D}^{n}}\left( \int_{\mathbb{R}^{m}}F^{J}\left(
x\right) ^{s}J_{\mu ,\theta }\left( x\right) dx\right) ^{\frac{r}{s}%
}\right\} ^{\frac{s}{r}}\leq \int_{\mathbb{R}^{m}}\left\{ \sum_{J\in 
\mathcal{D}^{n}}\left( F^{J}\left( x\right) ^{s}\right) ^{\frac{r}{s}%
}\right\} ^{\frac{s}{r}}J_{\mu ,\theta }\left( x\right) dx=\int_{\mathbb{R}%
^{m}}\left\{ \sum_{J\in \mathcal{D}^{n}}J_{\mu ,\theta }\left( x\right) ^{%
\frac{r}{s}}F^{J}\left( x\right) ^{r}\right\} ^{\frac{s}{r}}dx.
\end{equation*}%
Now apply Lemma \ref{theta bump lemma} above with the locally finite
absolutely continuous measures $\mu _{x}$ on $\mathbb{R}^{n}$ for a.e. $x\in 
\mathbb{R}^{m}$ to obtain%
\begin{eqnarray*}
\sum_{J\in \mathcal{D}^{n}}J_{\mu ,\theta }\left( x\right) ^{\frac{r}{s}%
}F^{J}\left( x\right) ^{r} &=&\sum_{J\in \mathcal{D}^{n}}J_{\mu ,\theta
}\left( x\right) ^{\frac{r}{s}}\left( \frac{1}{\left\vert J\right\vert _{\mu
_{x},\theta }}\int_{J}f_{x}\left( y\right) u_{x}\left( y\right) dy\right)
^{r} \\
&=&\sum_{J\in \mathcal{D}^{n}}\left\vert J\right\vert _{\mu _{x},\theta }^{%
\frac{r}{s}}\left( \frac{1}{\left\vert J\right\vert _{\mu _{x},\theta }}%
\int_{J}f_{x}\left( y\right) u_{x}\left( y\right) dy\right) ^{r} \\
&\lesssim &\left( \int_{\mathbb{R}^{n}}f_{x}\left( y\right) ^{s}u_{x}\left(
y\right) dy\right) ^{\frac{r}{s}}=\left( \int_{\mathbb{R}^{n}}f\left(
x,y\right) ^{s}u\left( x,y\right) dy\right) ^{\frac{r}{s}},
\end{eqnarray*}%
uniformly for a.e. $x\in \mathbb{R}^{m}$. Plugging this into the previous
display gives%
\begin{eqnarray*}
\left\{ \sum_{J\in \mathcal{D}^{n}}\left( \int_{\mathbb{R}^{m}}F^{J}\left(
x\right) ^{s}J_{\mu ,\theta }\left( x\right) dx\right) ^{\frac{r}{s}%
}\right\} ^{\frac{s}{r}} &\lesssim &\int_{\mathbb{R}^{m}}\left\{ \left(
\int_{\mathbb{R}^{n}}f\left( x,y\right) ^{s}u\left( x,y\right) dy\right) ^{%
\frac{r}{s}}\right\} ^{\frac{s}{r}}dx \\
&=&\int_{\mathbb{R}^{m}}\int_{\mathbb{R}^{n}}f\left( x,y\right) ^{s}u\left(
x,y\right) dydx=\left\Vert f\right\Vert _{L^{s}\left( \mu \right) }^{s}.
\end{eqnarray*}%
Altogether then we have%
\begin{eqnarray*}
&&\sum_{R\in \mathcal{R}^{m,n}}\left\vert R\right\vert _{\mu ,\theta }^{%
\frac{r}{s}}\left( \frac{1}{\left\vert R\right\vert _{\mu ,\theta }}%
\int_{R}f\left( x,y\right) u\left( x,y\right) dxdy\right) ^{r} \\
&\lesssim &\sum_{J\in \mathcal{D}^{n}}\left( \int_{\mathbb{R}%
^{m}}F^{J}\left( x\right) ^{s}J_{\mu ,\theta }\left( x\right) dx\right) ^{%
\frac{r}{s}}\lesssim \left\Vert f\right\Vert _{L^{s}\left( \mu \right)
}^{r}\ .
\end{eqnarray*}
\end{proof}

\subsubsection{Product fractional integrals}

The Tanaka-Yabuta theorem \cite[Theorem 1.1]{TaYa}, as well as the variant
in Theorem \ref{variation} above, uses an arbitrary nonnegative function $%
K\left( R\right) $ defined on dyadic rectangles $R\in \mathcal{R}^{m,n}$. If
for $0<\frac{\alpha }{m},\frac{\beta }{n}<1$, we define 
\begin{equation}
K_{\alpha ,\beta }^{m,n}\left( R\right) =K\left( I\times J\right) \equiv
\left\vert I\right\vert ^{\frac{\alpha }{m}-1}\left\vert J\right\vert ^{%
\frac{\beta }{n}-1},  \label{def K}
\end{equation}%
for $R=I\times J\in \mathcal{R}^{m,n}$, then in the special case $%
K=K_{\alpha ,\beta }^{m,n}$ we have the following pointwise estimate,%
\begin{eqnarray*}
\sum_{R\in \mathcal{R}^{m,n}}K_{\alpha ,\beta }^{m,n}\left( R\right) \mathbf{%
1}_{R}\left( x,y\right) \mathbf{1}_{R}\left( u,v\right) &=&\sum_{I\times
J\in \mathcal{R}^{m,n}}\left\{ K\left( I\times J\right) :x,u\in I\text{ and }%
y,v\in J\right\} \\
&=&\sum_{I\times J\in \mathcal{R}^{m,n}}\left\{ \left\vert I\right\vert ^{%
\frac{\alpha }{m}-1}\left\vert J\right\vert ^{\frac{\beta }{n}-1}:x,u\in I%
\text{ and }y,v\in J\right\} \\
&=&\sum_{I\in \mathcal{D}^{m}}\left\{ \left\vert I\right\vert ^{\frac{\alpha 
}{m}-1}:x,u\in I\right\} \ \times \ \sum_{J\in \mathcal{D}^{n}}\left\{
\left\vert J\right\vert ^{\frac{\beta }{n}-1}:y,v\in J\right\} \\
&\approx &d\left( x,u\right) ^{\frac{\alpha }{m}-1}d\left( y,v\right) ^{%
\frac{\beta }{n}-1}\lesssim \left\vert x-u\right\vert ^{\frac{\alpha }{m}%
-1}\left\vert y-v\right\vert ^{\frac{\beta }{n}-1},
\end{eqnarray*}%
where $d_{\limfunc{dy}}\left( x,u\right) $ denotes the \emph{dyadic}
distance between $x$ and $u$ in $\mathbb{R}^{m}$, and $d_{\limfunc{dy}%
}\left( y,v\right) $ denotes the \emph{dyadic} distance between $y$ and $v$
in $\mathbb{R}^{n}$. Here the dyadic distance between two points $p$ and $q$
in $\mathbb{R}^{k}$ is defined to be the side length of the smallest dyadic
cube containing $p$ and $q$. Note that the dyadic distance is at least $%
\frac{1}{\sqrt{k}}$ times the Euclidean distance since any dyadic cube $Q$
containing $x$ and $y$ must satisfy 
\begin{equation*}
\ell \left( Q\right) \geq \max_{1\leq i\leq k}\left\vert
x_{i}-y_{i}\right\vert \geq \sqrt{\frac{1}{k}\sum_{i=1}^{k}\left\vert
x_{i}-y_{i}\right\vert ^{2}}=\frac{1}{\sqrt{k}}\left\vert x-y\right\vert .
\end{equation*}%
So in order to apply the next theorem to the product fractional integral
operator with kernel $\left\vert x-u\right\vert ^{\frac{\alpha }{m}%
-1}\left\vert y-v\right\vert ^{\frac{\beta }{n}-1}$ it suffices to appeal to
Stromberg's well-known $\frac{1}{3}$-trick for the dyadic grids $\left\{ 
\mathcal{D}_{i}^{m}\right\} _{i=1}^{3^{m}}$ and $\left\{ \mathcal{D}%
_{j}^{n}\right\} _{j=1}^{3^{n}}$, to obtain%
\begin{equation}
\sum_{i=1}^{3^{m}}\sum_{j=1}^{3^{n}}\left[ \sum_{R=I\times J\in \mathcal{D}%
_{i}^{m}\times \mathcal{D}_{j}^{n}}K\left( R\right) \mathbf{1}_{R}\left(
x,y\right) \mathbf{1}_{R}\left( u,v\right) \right] \approx \left\vert
x-u\right\vert ^{\frac{\alpha }{m}-1}\left\vert y-v\right\vert ^{\frac{\beta 
}{n}-1}.  \label{expectation}
\end{equation}%
Variants of the following lemma can be found many times over in the
literature, too numerous to mention here. Let $\mathcal{P}^{N}$ denote the
collection of all cubes in $\mathbb{R}^{N}$ with sides parallel to the
coordinate axes.

\begin{lemma}
For $K\left( R\right) $ defined as in (\ref{def K}) we have (\ref%
{expectation}).
\end{lemma}

\begin{proof}
For convenience we recall a variation on the $\frac{1}{3}$-trick given in
Lemma 2.5 of \cite{HyLaPe}. For a given dyadic grid $\mathcal{D\subset P}%
^{N} $ with side lengths in $\left\{ \frac{2^{m}}{3}\right\} _{m\in \mathbb{Z%
}}$, partition the collection of tripled cubes $\left\{ 3I\right\} _{I\in 
\mathcal{D}}$ into $3^{N}$ subcollections $\left\{ S_{u}\right\}
_{u=1}^{3^{N}}$, with the property that for each subcollection $S_{u}$ there
exists a dyadic grid $\mathcal{D}_{u}$ with side lengths in $\left\{
2^{m}\right\} _{m\in \mathbb{Z}}$, such that $S_{u}\subset \mathcal{D}_{u}$.
With these grids $\left\{ \mathcal{D}_{u}\right\} _{u=1}^{3^{N}}$ fixed, we
have the following sandwiching property. For each cube $P\in \mathcal{P}^{N}$
and each integer $j\in \mathbb{N}$, there is a choice of $u=u\left(
P,j\right) $ with $1\leq u\leq 3^{n}$ and a cube $I=I_{u\left( P,j\right)
}\in \mathcal{D}_{u}$ such that 
\begin{eqnarray}
\ell \left( I\right) &\leq &18\ \ell \left( P\right) ,  \label{sandwich} \\
3P &\subset &I,  \notag \\
2^{j}P &\subset &\pi _{\mathcal{D}_{u}}^{\left( j\right) }I,  \notag
\end{eqnarray}%
where $\pi _{\mathcal{D}_{u}}^{\left( j\right) }I$ denotes the $j^{th}$
grandparent of $I$ in the grid $\mathcal{D}_{u}$.

Now fix $\left( x,y\right) \in \mathbb{R}^{m}\times \mathbb{R}^{n}$. For $%
x\in \mathbb{R}^{N}$, let $P\left( x,\ell \right) $ denote the cube centered
at $x$ with side length $\ell \in \left\{ 2^{k}\right\} _{k\in \mathbb{Z}}$.
Then with $R_{a,b}\left( x,y\right) \equiv P\left( x,2^{a}\right) \times
Q\left( y,2^{b}\right) $ for $a,b\in \mathbb{Z}$, we note that the right
hand side of (\ref{expectation}) is equivalent to%
\begin{equation*}
\sum_{a,b\in \mathbb{Z}}K\left( R_{a,b}\left( x,y\right) \right) \mathbf{1}%
_{R_{a,b}\left( x,y\right) }\left( x,y\right) \mathbf{1}_{R_{a,b}\left(
x,y\right) }\left( u,v\right) ,\ \ \ \ \ \left( u,v\right) \in \mathbb{R}%
^{m}\times \mathbb{R}^{n}.
\end{equation*}%
The first two lines in (\ref{sandwich}) now prove (\ref{expectation}), since
for each rectangle $R_{a,b}\left( x,y\right) \equiv P\left( x,2^{a}\right)
\times Q\left( y,2^{b}\right) $ there is $I\times J\in
\bigcup_{i=1}^{3^{m}}\bigcup_{j=1}^{3^{n}}\left( \mathcal{D}_{i}^{m}\times 
\mathcal{D}_{j}^{n}\right) $ such that 
\begin{equation*}
3R_{a,b}\left( x,y\right) \subset I\times J\subset 18R_{a,b}\left(
x,y\right) ,
\end{equation*}%
and moreover, by the definition of $K$ in (\ref{def K}), we then have $%
K\left( R_{a,b}\left( x,y\right) \right) \approx K\left( I\times J\right) $.
We do not need the third line in (\ref{expectation}) here.
\end{proof}

\begin{corollary}
\label{ext}Let $1<p<q<\infty $, $0<\alpha <m$,$\ 0<\beta <n$, $\theta >1$,
and let $v$ and $w$ be absolutely continuous weights on $\mathbb{R}%
^{m}\times \mathbb{R}^{n}$. Then the product fractional integral $I_{\alpha
,\beta }^{m,n}$ is bounded from $L^{p}\left( v^{p}\right) $ to $L^{q}\left(
w^{q}\right) $ if the $\theta $-bump rectangle characteristic $\mathbb{A}%
_{p,q;\theta }^{\left( \alpha ,\beta \right) ,\left( m,n\right) }\left(
v,w\right) $ is finite, where%
\begin{equation*}
\mathbb{A}_{p,q;\theta }^{\left( \alpha ,\beta \right) ,\left( m,n\right)
}\left( v,w\right) \equiv \sup_{I\times J\in \mathcal{R}^{m,n}}\left\vert
I\right\vert ^{\frac{\alpha }{m}-\frac{1}{p}+\frac{1}{q}}\left\vert
J\right\vert ^{\frac{\beta }{n}-\frac{1}{p}+\frac{1}{q}}\left( \frac{1}{%
\left\vert I\times J\right\vert }\int \int_{I\times J}v^{-p^{\prime }\theta
}\right) ^{\frac{1}{p^{\prime }\theta }}\ \left( \frac{1}{\left\vert I\times
J\right\vert }\int \int_{I\times J}w^{q\theta }\right) ^{\frac{1}{q\theta }%
}\ .
\end{equation*}
\end{corollary}

\begin{remark}
The above proof of the Corollary, when restricted to the $1$-parameter case,
gives a short and elegant proof of Theorem 1(A) in \cite{SaWh} in the
special case $p<q$.
\end{remark}

\subsection{Reverse doubling weights for bilinear embeddings}

Here is a slight improvement of the theorem of Tanaka and Yabuta \cite{TaYa}%
, valid for the product fractional integral kernel, as well as more general
kernels $K$ satisfying property (\ref{satisfy both}) below regarding
expectations taken over partial grids $\mathcal{R}^{m,n}=\mathcal{D}%
^{m}\times \mathcal{D}^{n}$. Recall that $\mu $ is a product reverse
doubling weight on $\mathbb{R}^{m}\times \mathbb{R}^{n}$ if (\ref{product
rev doub}) holds.

\begin{theorem}
\label{theorem rev doub}Suppose $1<p<q<\infty $. Let $\sigma $ and $\omega $
be product reverse doubling weights on $\mathbb{R}^{m}\times \mathbb{R}^{n}$%
, and let $K=K_{\alpha ,\beta }^{m,n}:\mathcal{R}^{m,n}\rightarrow \left[
0,\infty \right) $ be as in (\ref{def K}), or more generally satisfy the
expectation inequality (\ref{satisfy both}) below. Then the norm $\mathbb{N}%
_{K}\left( \sigma ,\omega \right) $ of the positive bilinear inequality,%
\begin{equation*}
\sum_{R\in \mathcal{R}^{m,n}}K\left( R\right) \left( \int_{R}fd\sigma
\right) \left( \int_{R}gd\omega \right) \leq \mathbb{N}_{K}\left( \sigma
,\omega \right) \ \left\Vert f\right\Vert _{L^{p}\left( \sigma \right)
}\left\Vert g\right\Vert _{L^{q^{\prime }}\left( \omega \right) }\ ,\ \ \ \
\ f,g\geq 0,
\end{equation*}%
is finite for all partial grids $\mathcal{R}^{m,n}=\mathcal{D}^{m}\times 
\mathcal{D}^{n}$ \emph{if and only if}%
\begin{equation*}
\mathbb{A}_{K}\left( \sigma ,\omega \right) \equiv \sup_{R\in \mathcal{P}%
^{m}\times \mathcal{P}^{n}}K\left( R\right) \ \left\vert R\right\vert
_{\omega }^{\frac{1}{q}}\ \left\vert R\right\vert _{\sigma }^{\frac{1}{%
p^{\prime }}}<\infty ,\ \ \ \ \ \text{for all rectangles }R\in \mathcal{P}%
^{n}\times \mathcal{P}^{m}\ .
\end{equation*}
\end{theorem}

\begin{proof}
We begin the proof with a brief review of the good/bad grid technology of
Nazarov, Treil and Volberg. See \cite{NTV2}, \cite{NTV4}, or \cite{Vol} for
more detail. We restrict to dimension $n=1$ for the moment. Let $%
0<\varepsilon <1$ and $\mathbf{r}\in \mathbb{N}$ to be chosen later. Define $%
J$ to be $\varepsilon -\limfunc{good}$ in an interval $K$ if 
\begin{equation*}
d\left( J,\limfunc{skel}K\right) >2\left\vert J\right\vert ^{\varepsilon
}\left\vert K\right\vert ^{1-\varepsilon },
\end{equation*}%
where the skeleton $\limfunc{skel}K$ of an interval $K$ consists of its two
endpoints and its midpoint. Define $\mathcal{D}_{\left( \mathbf{r}%
,\varepsilon \right) -\limfunc{good}}$ to consist of those $J\in \mathcal{D}$
such that $J$ is good in every superinterval $K\in \mathcal{D}$ that lies at
least $\mathbf{r}$ levels above $J$. As the goodness parameters $\varepsilon 
$ and $\mathbf{r}$ will eventually be fixed throughout the proof, we
sometimes suppress the parameters, and simply write $\mathcal{D}_{\limfunc{%
good}}$ in place of $\mathcal{D}_{\left( \mathbf{r},\varepsilon \right) -%
\limfunc{good}}$, and say "$J$ is $\limfunc{good}$" instead of "$J$ is good
in every superinterval $K\in \mathcal{D}$ that lies at least $\mathbf{r}$
levels above $J$". We also define $\mathcal{D}_{\limfunc{bad}}\equiv 
\mathcal{D}\setminus \mathcal{D}_{\limfunc{good}}$.

\textbf{Parameterizations of dyadic grids}: Here we recall a construction
from \cite{SaShUr10} that was in turn based on that of Hyt\"{o}nen in \cite%
{Hyt2}. Momentarily fix a large positive integer $M\in \mathbb{N}$, and
consider the tiling of $\mathbb{R}$ by the family of intervals $\mathbb{D}%
_{M}\equiv \left\{ I_{\alpha }^{M}\right\} _{\alpha \in \mathbb{Z}}$ having
side length $2^{-M}$ and given by $I_{\alpha }^{M}\equiv
I_{0}^{M}+2^{-M}\alpha $ where $I_{0}^{M}=\left[ 0,2^{-M}\right) $. A \emph{%
dyadic grid} $\mathcal{D}$ built on $\mathbb{D}_{M}$ is\ defined to be a
family of intervals $\mathcal{D}$ satisfying:

(\textbf{1}) Each $I\in \mathcal{D}$ has side length $2^{-\ell }$ for some $%
\ell \in \mathbb{Z}$ with $\ell \leq M$, and $I$ is a union of $2^{M-\ell }$
intervals from the tiling $\mathbb{D}_{M}$,

(\textbf{2}) For $\ell \leq M$, the collection $\mathcal{D}_{\ell }$ of
intervals in $\mathcal{D}$ having side length $2^{-\ell }$ forms a pairwise
disjoint decomposition of the space $\mathbb{R}$,

(\textbf{3}) Given $I\in \mathcal{D}_{i}$ and $J\in \mathcal{D}_{j}$ with $%
j\leq i\leq M$, it is the case that either $I\cap J=\emptyset $ or $I\subset
J$.

We now momentarily fix a \emph{negative} integer $N\in -\mathbb{N}$, and
restrict the above grids to intervals of side length at most $2^{-N}$:%
\begin{equation*}
\mathcal{D}^{N}\equiv \left\{ I\in \mathcal{D}:\text{side length of }I\text{
is at most }2^{-N}\right\} \text{.}
\end{equation*}%
We refer to such grids $\mathcal{D}^{N}$ as a (truncated) dyadic grid $%
\mathcal{D}$ built on $\mathbb{D}_{M}$ of size $2^{-N}$. There are now two
traditional means of constructing probability measures on collections of
such dyadic grids, namely parameterization by choice of parent, and
parameterization by translation. We will only need the former
parameterization here. For any 
\begin{equation*}
\beta =\{\beta _{i}\}_{i\in _{M}^{N}}\in \omega _{M}^{N}\equiv \left\{
0,1\right\} ^{\mathbb{Z}_{M}^{N}},
\end{equation*}%
where $\mathbb{Z}_{M}^{N}\equiv \left\{ \ell \in \mathbb{Z}:N\leq \ell \leq
M\right\} $, define the dyadic grid $\mathcal{D}_{\beta }$ built on $\mathbb{%
D}_{m}$ of size $2^{-N}$ by 
\begin{equation}
\mathcal{D}_{\beta }=\left\{ 2^{-\ell }\left( [0,1)+k+\sum_{i:\ \ell <i\leq
m}2^{-i+\ell }\beta _{i}\right) \right\} _{N\leq \ell \leq m,\,k\in {\mathbb{%
Z}}}\ .  \label{def dyadic grid}
\end{equation}%
Place the uniform probability measure $\rho _{M}^{N}$ on the finite index
space $\omega _{M}^{N}=\left\{ 0,1\right\} ^{\mathbb{Z}_{M}^{N}}$, namely
that which charges each $\beta \in \omega _{M}^{N}$ equally.

This construction may be thought of as being \emph{parameterized by scales}
- each component $\beta _{i}$ in $\beta =\{\beta _{i}\}_{i\in _{M}^{N}}\in
\omega _{M}^{N}$ amounting to a choice of the two possible tilings at level $%
i$ that respect the choice of tiling at the level below. For purposes of
notation and clarity, we now suppress all reference to $M$ and $N$ in our
families of grids, and use the notation $\Omega $ instead of $\omega
_{M}^{N} $ for the index or parameter set, and then use $\boldsymbol{P}%
_{\Omega }$ and $\boldsymbol{E}_{\Omega }$ to denote probability and
expectation with respect to families of grids. We will now instead proceed
as if all grids considered are unrestricted. The careful reader can supply
the modifications necessary to handle the assumptions made above on the
grids $\mathcal{D}$ regarding $M$ and $N$.

Given a pair of grids $\mathcal{D}^{m}$ and $\mathcal{D}^{n}$ in $\mathbb{R}%
^{m}$ and $\mathbb{R}^{n}$ respectively, form the corresponding partial grid 
$\mathcal{R}^{m,n}=\mathcal{D}^{m}\times \mathcal{D}^{n}$ of rectangles. We
say that a rectangle $R=I\times J\in \mathcal{R}_{\limfunc{good}}^{m,n}$
(and say $R$ is good) if both $I\in \mathcal{D}_{\limfunc{good}}^{m}$ and $%
J\in \mathcal{D}_{\limfunc{good}}^{n}$. Given a positive bilinear form%
\begin{equation*}
\mathcal{B}_{\mathcal{R}^{m,n}}\left( f,g\right) \equiv \sum_{R\in \mathcal{R%
}^{m,n}}K\left( R\right) \left( \int_{R}fd\sigma \right) \left(
\int_{R}gd\omega \right) ,\ \ \ \ \ f\in L^{p}\left( \sigma \right) ,g\in
L^{q^{\prime }}\left( \omega \right) ,
\end{equation*}%
we follow the NTV idea and dominate $B_{\mathcal{R}^{m,n}}\left( f,g\right)
=B_{\mathcal{D}^{m}\times \mathcal{D}^{n}}\left( f,g\right) $ as follows:%
\begin{eqnarray*}
\mathcal{B}_{\mathcal{D}^{m}\times \mathcal{D}^{n}}\left( f,g\right) &\leq
&\left\{ \sum_{I\times J\in \mathcal{D}_{\limfunc{good}}^{m}\times \mathcal{D%
}_{\limfunc{good}}^{n}}+\sum_{I\times J\in \mathcal{D}^{m}\times \mathcal{D}%
_{\limfunc{bad}}^{n}}+\sum_{I\times J\in \mathcal{D}_{\limfunc{bad}%
}^{m}\times \mathcal{D}^{n}}\right\} K\left( I\times J\right) \left(
\int_{I\times J}fd\sigma \right) \left( \int_{I\times J}gd\omega \right) \\
&\equiv &\mathcal{B}_{\mathcal{D}_{\limfunc{good}}^{m}\times \mathcal{D}_{%
\limfunc{good}}^{n}}\left( f,g\right) +\mathcal{B}_{\mathcal{D}^{m}\times 
\mathcal{D}_{\limfunc{bad}}^{n}}\left( f,g\right) +\mathcal{B}_{\mathcal{D}_{%
\limfunc{bad}}^{m}\times \mathcal{D}^{n}}\left( f,g\right) .
\end{eqnarray*}%
From the previous subsection we have that the positive bilinear form%
\begin{equation*}
\mathcal{I}\left( f,g\right) \equiv \int_{\mathbb{R}^{m}\times \mathbb{R}%
^{n}}I_{\alpha ,\beta }^{m,n}\left( f\sigma \right) g\omega
\end{equation*}%
satisfies%
\begin{equation}
\boldsymbol{E}_{\Omega \times \Omega }\mathcal{B}_{\mathcal{D}^{m}\times 
\mathcal{D}^{n}}\left( f,g\right) \geq c\mathcal{B}_{\mathcal{D}^{m}\times 
\mathcal{D}^{n}}\left( f,g\right) ,\ \ \ \ \ \text{for all }\mathcal{D}%
^{m}\times \mathcal{D}^{n}\text{ and some }c>0.  \label{satisfy both}
\end{equation}%
It then follows that the norm $\mathfrak{N}_{\mathcal{I}}$ of the bilinear
form $\mathcal{I}$ can be estimated using $\left\Vert f\right\Vert
_{L^{p}\left( \sigma \right) }=\left\Vert g\right\Vert _{L^{q^{\prime
}}\left( \omega \right) }=1$ chosen so that $\mathfrak{N}_{\mathcal{I}}=%
\mathcal{I}\left( f,g\right) $:%
\begin{eqnarray*}
\mathfrak{N}_{\mathcal{I}} &=&\mathcal{I}\left( f,g\right) =\boldsymbol{E}%
_{\Omega \times \Omega }\mathcal{B}_{\mathcal{D}^{m}\times \mathcal{D}%
^{n}}\left( f,g\right) \\
&\leq &\boldsymbol{E}_{\Omega \times \Omega }\mathcal{B}_{\mathcal{D}_{%
\limfunc{good}}^{m}\times \mathcal{D}_{\limfunc{good}}^{n}}\left( f,g\right)
+\boldsymbol{E}_{\Omega \times \Omega }\mathcal{B}_{\mathcal{D}^{m}\times 
\mathcal{D}_{\limfunc{bad}}^{n}}\left( f,g\right) +\boldsymbol{E}_{\Omega
\times \Omega }\mathcal{B}_{\mathcal{D}_{\limfunc{bad}}^{m}\times \mathcal{D}%
^{n}}\left( f,g\right) .
\end{eqnarray*}

Now the conditional probability that a given cube $I$ is bad in a grid $%
\mathcal{D}^{m}$ that contains it is small, in fact (see e.g. \cite{NTV2}, 
\cite{NTV4}, \cite{Vol} or \cite[Subsubsection 3.1.1]{SaShUr}) we have 
\begin{equation*}
\boldsymbol{P}_{\Omega }\left\{ \mathcal{D}^{m}:I\text{ is bad in }\mathcal{D%
}^{m}\mid \text{conditioned on }I\in \mathcal{D}^{m}\right\} \leq
C2^{-\varepsilon \mathbf{r}}.
\end{equation*}%
Thus we obtain%
\begin{eqnarray*}
\boldsymbol{E}_{\Omega \times \Omega }\mathcal{B}_{\mathcal{D}^{m}\times 
\mathcal{D}_{\limfunc{bad}}^{n}}\left( f,g\right) &\leq &C2^{-\varepsilon 
\mathbf{r}}\mathfrak{N}_{\mathcal{I}}\left\Vert f\right\Vert _{L^{p}\left(
\sigma \right) }\left\Vert g\right\Vert _{L^{q^{\prime }}\left( \omega
\right) }=C2^{-\varepsilon \mathbf{r}}\mathfrak{N}_{\mathcal{I}}\ , \\
\boldsymbol{E}_{\Omega \times \Omega }\mathcal{B}_{\mathcal{D}_{\limfunc{bad}%
}^{m}\times \mathcal{D}^{n}}\left( f,g\right) &\leq &C2^{-\varepsilon 
\mathbf{r}}\mathfrak{N}_{\mathcal{I}}\left\Vert f\right\Vert _{L^{p}\left(
\sigma \right) }\left\Vert g\right\Vert _{L^{q^{\prime }}\left( \omega
\right) }=C2^{-\varepsilon \mathbf{r}}\mathfrak{N}_{\mathcal{I}}\ ,
\end{eqnarray*}%
and hence%
\begin{equation*}
\mathfrak{N}_{\mathcal{I}}\leq \boldsymbol{E}_{\Omega \times \Omega }%
\mathcal{B}_{\mathcal{D}_{\limfunc{good}}^{m}\times \mathcal{D}_{\limfunc{%
good}}^{n}}\left( f,g\right) +2C2^{-\varepsilon \mathbf{r}}\mathfrak{N}_{%
\mathcal{I}}\ ,
\end{equation*}%
which gives%
\begin{equation*}
\mathfrak{N}_{\mathcal{I}}\leq \frac{1}{1-2C2^{-\varepsilon \mathbf{r}}}%
\boldsymbol{E}_{\Omega \times \Omega }\mathcal{B}_{\mathcal{D}_{\limfunc{good%
}}^{m}\times \mathcal{D}_{\limfunc{good}}^{n}}\left( f,g\right)
\end{equation*}%
if $\varepsilon \mathbf{r}$ is chosen sufficiently small.

Thus we see that in order to prove Theorem \ref{theorem rev doub}, we need
only consider the `good' bilinear form $\mathcal{B}_{\mathcal{D}_{\limfunc{%
good}}^{m}\times \mathcal{D}_{\limfunc{good}}^{n}}\left( f,g\right) $ and
estimate it independently of the partial grid of good rectangles $\mathcal{D}%
_{\limfunc{good}}^{m}\times \mathcal{D}_{\limfunc{good}}^{n}$. Then using
arguments as in \cite{TaYa} or above, the proof of Theorem \ref{theorem rev
doub} is reduced to the following Carleson embedding for `good' rectangles.

\textbf{Carleson embedding}: Suppose that $1<s<r<\infty $ and that $\mu $ is
a product reverse doubling measure on $\mathbb{R}^{m}\times \mathbb{R}^{n}$.
Then we have%
\begin{equation*}
\left\{ \sum_{R\in \mathcal{D}_{\limfunc{good}}^{m}\times \mathcal{D}_{%
\limfunc{good}}^{n}}\left\vert R\right\vert _{\mu }^{\frac{r}{s}}\left( 
\frac{1}{\left\vert R\right\vert _{\mu }}\int_{R}fd\mu \right) ^{r}\right\}
^{\frac{1}{r}}\leq C_{s,r}\left\Vert f\right\Vert _{L^{s}\left( \mu \right)
}\ ,\ \ \ \ \ f\geq 0,
\end{equation*}%
where $C_{s,r}$ depends only on $s$, $r$, the reverse doubling constants for 
$\mu $, and the goodness parameters $\varepsilon ,\mathbf{r}$. In paricular, 
$C_{s,r}$ is independent of the partial grid $\mathcal{D}_{\limfunc{good}%
}^{m}\times \mathcal{D}_{\limfunc{good}}^{n}$. Continuing to follow the
iteration argument of Tanaka and Yabuta as in \cite{TaYa} or above, further
reduces matters to proving the following Carleson condition on cubes for a
reverse doubling measure $\mu $ on $\mathbb{R}^{N}$ with exponent $\eta >0$,
and a power $\rho >1$:%
\begin{equation}
\sum_{Q\in \mathcal{D}_{\limfunc{good}}^{N}:\ Q\subset P}\left\vert
Q\right\vert _{\mu }^{\rho }\leq C_{N,\mathbf{r},\varepsilon ,\rho
}\left\vert P\right\vert _{\mu ,\theta }^{\rho }\ .  \label{Car rev doub}
\end{equation}%
Indeed, the reader can easily verify that the arguments work just as well
for the subgrids $\mathcal{D}_{\limfunc{good}}^{m}$ and $\mathcal{D}_{%
\limfunc{good}}^{n}$ in place of the grids $\mathcal{D}^{m}$ and $\mathcal{D}%
^{n}$.

It is now at this point that the goodness of the cubes $Q$ plays a crucial
role in conjuction with the reverse doubling property. To see (\ref{Car rev
doub}), recall the goodness parameters $0<\varepsilon <1$ and $\mathbf{r}\in 
\mathbb{N}$ and observe that if $Q$ is a good cube contained in $P$ then%
\newline
\textbf{either} $\ell \left( Q\right) \geq \ell \left( P\right) -\mathbf{r}$
and we can use the trivial estimate $\left\vert Q\right\vert _{\mu }^{\rho
}\leq \left\vert P\right\vert _{\mu }^{\rho }$,\newline
\textbf{or} $\ell \left( Q\right) <\ell \left( P\right) -r$ in which case $%
\limfunc{dist}\left( Q,\partial P\right) \geq 2\ell \left( Q\right)
^{\varepsilon }\ell \left( P\right) ^{1-\varepsilon }$.

In this latter case we note that if $\ell \left( Q\right) =2^{-k}\ell \left(
P\right) $ then 
\begin{equation*}
2^{k\left( 1-\varepsilon \right) }Q=\left( \frac{\ell \left( P\right) }{\ell
\left( Q\right) }\right) ^{1-\varepsilon }Q\subset \frac{2\ell \left(
Q\right) ^{\varepsilon }\ell \left( P\right) ^{1-\varepsilon }}{\ell \left(
Q\right) }Q\subset \frac{\limfunc{dist}\left( Q,\partial P\right) }{\ell
\left( Q\right) }Q\subset P
\end{equation*}%
and so by reverse doubling we have 
\begin{equation*}
\left\vert Q\right\vert _{\mu }\leq C2^{-\eta k\left( 1-\varepsilon \right)
}\left\vert \left( \frac{\ell \left( P\right) }{\ell \left( Q\right) }%
\right) ^{1-\varepsilon }Q\right\vert _{\mu }\leq C2^{-\eta \left(
1-\varepsilon \right) k}\left\vert P\right\vert _{\mu }\ .
\end{equation*}%
Thus we can estimate%
\begin{eqnarray*}
\sum_{Q\in \mathcal{D}_{\limfunc{good}}^{N}:\ Q\subset P}\left\vert
Q\right\vert _{\mu }^{\rho } &=&\sum_{k=0}^{\mathbf{r}}2^{N\mathbf{r}%
}\left\vert P\right\vert _{\mu }^{\rho }+\sum_{k=\mathbf{r}+1}^{\infty
}\sum_{Q\in \mathcal{D}_{\limfunc{good}}^{N}:\ \ell \left( Q\right)
=2^{-k}\ell \left( P\right) }\left\vert Q\right\vert _{\mu }^{\rho
-1}\left\vert Q\right\vert _{\mu } \\
&\leq &C_{N,\mathbf{r}}\left\vert P\right\vert _{\mu }^{\rho }+\sum_{k=%
\mathbf{r}+1}^{\infty }\sum_{Q\in \mathcal{D}_{\limfunc{good}}^{N}:\ \ell
\left( Q\right) =2^{-k}\ell \left( P\right) }\left( C2^{-\eta \left(
1-\varepsilon \right) k}\left\vert P\right\vert _{\mu }\right) ^{\rho
-1}\left\vert Q\right\vert _{\mu } \\
&\leq &C_{N,\mathbf{r}}\left\vert P\right\vert _{\mu }^{\rho }+\left\{
\sum_{k=0}^{\infty }\left( C2^{-\eta \left( 1-\varepsilon \right) \left(
\rho -1\right) k}\right) \right\} \left\vert P\right\vert _{\mu }^{\rho
}=C_{N,\mathbf{r},\varepsilon ,\rho }\left\vert P\right\vert _{\mu }^{\rho
}\ .
\end{eqnarray*}%
This completes the proof of (\ref{Car rev doub}), and hence also that of
Theorem \ref{theorem rev doub}.
\end{proof}

\subsection{Concluding remarks}

In the case of kernels $K=K_{\alpha ,\beta }^{m,n}$ given by (\ref{def K}),
or more generally that satisfy (\ref{satisfy both}), one can assume for each
weight separately, either rectangle reverse doubling, or a half $\theta $%
-bump condition, in order to obtain norm boundedness. For example, the
following hybrid theorem holds.

\begin{theorem}
Suppose $1<p<q<\infty $. Let $\sigma $ be a product reverse doubling weight
on $\mathbb{R}^{n}$, let $d\omega \left( x\right) =w\left( x\right) ^{q}dx$
be absolutely continuous with respect to Lebesgue measure, and let $%
K=K_{\alpha ,\beta }^{m,n}:\mathcal{R}^{m,n}\rightarrow \left[ 0,\infty
\right) $ be as in (\ref{def K}), or more generally satisfy (\ref{satisfy
both}). Then the norm $\mathbb{N}_{K}\left( \sigma ,\omega \right) $ of the
positive bilinear inequality,%
\begin{equation*}
\sum_{R\in \mathcal{R}^{m,n}}K\left( R\right) \left( \int_{R}fd\sigma
\right) \left( \int_{R}gd\omega \right) \leq \mathbb{N}_{K}\left( \sigma
,\omega \right) \ \left\Vert f\right\Vert _{L^{p}\left( \sigma \right)
}\left\Vert g\right\Vert _{L^{q^{\prime }}\left( \omega \right) }\ ,\ \ \ \
\ f,g\geq 0,
\end{equation*}%
is finite for all products of grids $\mathcal{R}^{m,n}=\mathcal{D}^{m}\times 
\mathcal{D}^{n}$ if the half $\theta $-bump rectangle characteristic $%
\mathbb{A}_{K,\theta }^{\omega }\left( \sigma ,\omega \right) $ is finite,
where%
\begin{eqnarray*}
\mathbb{A}_{K,\theta }^{\omega }\left( \sigma ,\omega \right) &\equiv
&\sup_{R\in \mathcal{R}^{m,n}}K\left( R\right) \ \left(
\int_{R}v^{-p^{\prime }}d\sigma \right) ^{\frac{1}{p^{\prime }}}\ \left[
\left\vert R\right\vert ^{\frac{1}{q\theta ^{\prime }}}\left(
\int_{R}w^{q\theta }d\omega \right) ^{\frac{1}{q\theta }}\right] \\
&=&\sup_{R\in \mathcal{R}^{m,n}}K\left( R\right) \ \left\vert R\right\vert
_{\omega ,\theta }^{\frac{1}{q}}\ \left\vert R\right\vert _{\sigma }^{\frac{1%
}{p^{\prime }}}\ .
\end{eqnarray*}
\end{theorem}

The proof is an easy exercise in combining the proofs of Theorems \ref%
{variation} and \ref{theorem rev doub} above.

\section{Appendix}

We say that a weight $\mu $ on the real line is \textbf{strongly} reverse
doubling if there is $\beta <1$ such that%
\begin{equation*}
\left\vert I_{\limfunc{left}}\right\vert _{\mu },\left\vert I_{\limfunc{right%
}}\right\vert _{\mu }\leq \beta \left\vert I\right\vert _{\mu }\text{ for
all intervals }I,
\end{equation*}%
where if $I=\left[ a,b\right) $, then $I_{\limfunc{left}}=\left[ a,\frac{a+b%
}{2}\right) $ and $I_{\limfunc{right}}=\left[ \frac{a+b}{2},b\right) $ are
the left and right halves of $I$ respectively. A strongly reverse doubling
weight on $\mathbb{R}$ is a doubling weight on $\mathbb{R}$, since if we
choose $N$ so large that $\beta ^{N}<\frac{1}{4}$, then for $I=\left[
a,b\right) $, we have 
\begin{equation*}
\left\vert \left[ a,a+\frac{b-a}{2^{N}}\right) \right\vert _{\mu },\ \ \
\left\vert \left[ b-\frac{b-a}{2^{N}},b\right) \right\vert _{\mu }\leq \beta
^{N}\left\vert I\right\vert _{\mu }<\frac{1}{4}\left\vert I\right\vert _{\mu
}\ .
\end{equation*}%
Hence%
\begin{eqnarray*}
\left\vert \left[ a+\frac{b-a}{2^{N}},b-\frac{b-a}{2^{N}}\right) \right\vert
_{\mu } &=&\left\vert \left[ a,b\right) \right\vert _{\mu }-\left\vert \left[
a,a+\frac{b-a}{2^{N}}\right) \right\vert _{\mu }-\left\vert \left[ b-\frac{%
b-a}{2^{N}},b\right) \right\vert _{\mu } \\
&\geq &\left( 1-\frac{1}{4}-\frac{1}{4}\right) \left\vert I\right\vert _{\mu
}=\frac{1}{2}\left\vert I\right\vert _{\mu }\ ,
\end{eqnarray*}%
where the length of the interval $\left[ a+\frac{b-a}{2^{N}},b-\frac{b-a}{%
2^{N}}\right) $ is $\frac{2^{N-1}-1}{2^{N-1}}\ell \left( I\right) $. Thus
with $\gamma =\frac{2^{N-1}}{2^{N-1}-1}>1$, we have for every interval $K$,%
\begin{equation*}
\left\vert \gamma K\right\vert _{\mu }\leq 2\left\vert K\right\vert _{\mu
},\ \text{hence }\left\vert 2K\right\vert _{\mu }\leq 2^{M}\left\vert
K\right\vert _{\mu }\text{ if }\gamma ^{M}\geq 2,
\end{equation*}%
which shows that $\mu $ is doubling. Similarly we see that a strongly
rectangle reverse doubling weight on $\mathbb{R}^{N}$ is a rectangle
doubling weight on $\mathbb{R}^{N}$. Here $\mu $ is strongly rectangle
reverse doubling if there is $\beta <1$ such that%
\begin{eqnarray*}
&&\left\vert I^{1}\times ...\times I_{\limfunc{left}}^{k}\times ...\times
I^{N}\right\vert _{\mu },\left\vert I^{1}\times ...\times I_{\limfunc{right}%
}^{k}\times ...\times I^{N}\right\vert _{\mu } \\
&\leq &\beta \left\vert I^{1}\times ...\times I_{\mu }^{N}\right\vert _{\mu }%
\text{ for all rectangles }I^{1}\times ...\times I^{N}\text{\ and }1\leq
k\leq N,
\end{eqnarray*}%
and $\mu $ is rectangle doubling if there is $C>0$ such that%
\begin{equation*}
\left\vert \left( 2I^{1}\right) \times ...\times \left( 2I^{N}\right)
\right\vert _{\mu }\leq C\left\vert I^{1}\times ...\times I^{N}\right\vert
_{\mu }\text{ for all rectangles }I^{1}\times ...\times I^{N}.
\end{equation*}

\begin{example}
Suppose that $\mu $ is a doubling weight on $\mathbb{R}^{N}$. Then $d\nu
\left( x\right) \equiv \mathbf{1}_{\left[ 0,\infty \right) ^{N}}\left(
x\right) \mu \left( x\right) $ is a reverse doubling weight on $\mathbb{R}%
^{N}$ that is not a doubling weight on $\mathbb{R}^{N}$.
\end{example}

\end{document}